\documentclass[reqno]{amsart}

\usepackage[usenames,dvipsnames]{pstricks}
\usepackage{epsfig}

\usepackage{color}
\date{\today}

\usepackage[latin1]{inputenc}
\usepackage[T1]{fontenc}
\usepackage{amsmath, amsthm}
\usepackage[english]{babel}
\usepackage{amssymb}
\usepackage{latexsym}
\usepackage{fancyhdr}
\usepackage{graphicx}
\usepackage[all]{xy}

\theoremstyle{defin}
\newtheorem{defin}{{\bf Definition}}[section]

\newtheorem{nota}[defin]{{\bf Remark}}
\newtheorem{exemplo}[defin]{{\bf Example}}

\newtheorem{teo}[defin]{{\bf Theorem}}
\newtheorem{prop}[defin]{{\bf Proposition}}

\newtheorem{corol}[defin]{{\bf Corollary}}

\newcommand{\ot}{\otimes}
\newcommand{\co}{\circ}

\begin{document}

\begin{center}
 {\huge{\bf  Equivalences for weak crossed products}}

\end{center}

\ \\

{\bf  J.M. Fern\'andez Vilaboa$^{1}$, R. Gonz\'{a}lez
Rodr\'{\i}guez$^{2}$ and A.B. Rodr\'{\i}guez Raposo$^{3}$}

\ \\
\hspace{-0,5cm}$1$ Departamento de \'Alxebra, Universidad de
Santiago de Compostela,  E-15771 Santiago de Compostela, Spain
(e-mail: josemanuel.fernandez@usc.es)
\ \\
\hspace{-0,5cm}$2$ Departamento de Matem\'{a}tica Aplicada II,
Universidad de Vigo, Campus Universitario Lagoas-Mar\-co\-sen\-de,
E-36310 Vigo, Spain (e-mail: rgon@dma.uvigo.es)
\\
\hspace{-0,5cm}$3$ Departamento de Matem\'aticas, Universidade da Coru\~{n}a,
Escuela Polit\'ecnica Superior, E-15403 Ferrol, Spain
(e-mail: abraposo@edu.xunta.es)
\ \\

\begin{center}
{\bf Abstract}
\end{center}
{\small In this paper we give a criterion that characterizes equivalent weak crossed products.  By duality, we obtain
a similar result for weak crossed coproducts and, as a consequence, we find the conditions that assures the
equivalence between two weak crossed biproducts. As an application, we show that the main results proved by Panaite in \cite{Pan3} (see also \cite{Pan1}), for Brzezi\'nski's crossed products,  admits a substantial reduction in the imposed conditions.}

\vspace{0.5cm}



\section*{Introduction}

In the developement and further generalizations of the theory of Hopf algebras, crossed products play an important role. In \cite{mra-preunit} the so-called weak crossed product is defined, and it happen to contain as examples several types of crossed product constructions. For example, unified crossed products studied by Agore and Militaru in \cite{AM2}, partial crossed products studied by Muniz {\it et. al.} in \cite{partial} or Brzezi\'nski crossed products \cite{tb-crpr} are particular instances of weak crossed products. Moreover, weak crossed products provide also a general setting for studying crossed products in weak contexts as weak wreath products \cite{wreath} or weak crossed products for weak bialgebras \cite{ana1}. When we dualize the notion of weak crossed product we obtain a weak crossed coproduct, that, under certain conditions, can be glued to a weak crossed product so a weak crossed biproduct is obtained \cite{mra-proj}. Weak crossed biproducts are a generalization of cross products given by Bespalov and Drabant \cite{bes-drab1}. 

Equivalences of Brzezi\'nski crossed products and cross products were characterized by Panaite in \cite{Pan1, Pan3}. In these papers, Panaite gives a definition for the equivalence of two crossed products $A\otimes_{R, \sigma}V$ and $A\otimes_{R', \sigma'}V$, where $R,\sigma, R'$ and $\sigma'$ are morphisms used to define the multiplications on $A\otimes V$. The equivalence between these products can be given in terms of two morphisms $\theta:V\rightarrow A\otimes V$ and $\gamma:V\rightarrow A\otimes V$ such that an explicit relation between $R, \sigma$ and $R', \sigma'$ is obtained. This idea was used by Brzezi\'nski in \cite{tb-crpr} to characterize equivalences of crossed products by a coalgebra, and was also used in \cite{ana1} to obtain equivalences between two weak crossed products of weak bialgebras. However, this last case is not included in Panaite's theory for being a crossed product in a weak context. By a dualization of his results, Panaite obtains a characterization of the equivalence of crossed coproducts, and thus, he obtains results related to equivalences of cross product bialgebras. 

Let ${\mathbb A}_V=(A, V, \psi_V^A, \sigma_V^A)$ and ${\mathbb A}_W=(A, V, \psi_W^A, \sigma_W^A)$ be two four-tuples where $A$ is a monoid, and $\psi_V^A, \sigma_V^A$ and $\psi_W^A, \sigma_W^A$ are the morphisms we use to define associative multiplications on $A\otimes V$ and $A\otimes W$ respectively, so we have two weak crossed products. In the present paper we give some necessary and sufficient conditions for these products to be equivalent. We find that the two weak crossed products are equivalent if, and only if, morphisms $\psi_W^A$ and $\sigma_W^A$ can be obtained from $\psi_V^A$ and $\sigma_V^A$ using two suitable morphisms $\gamma:V\rightarrow A\otimes W$ and $\theta:W\rightarrow A\otimes V$. This result can be extended to the case of weak crossed products with preunit. When we particularize to the non-weak case and take $V=W$, we recover the results of equivalence studied by Panaite in \cite{Pan1, Pan3}. Moreover, our result supposes a substantial reduction in the imposed conditions in \cite{Pan1, Pan3}. Also, the theory we present here can be used to improve results on equivalences of crossed products of weak bialgebras studied in \cite{ana1} and \cite{nrm-hhapl}. When we dualize our result, we find a characterization for equivalences between two weak crossed coproducts and, if we glue together both cases, we obtain a characterization of the equivalence between two weak crossed biproducts. This characterization also extends the one by Panaite in \cite{Pan3}.

Throughout this paper $\mathcal C$ denotes a strict  monoidal category with tensor product $\ot$, unit object $K$. There is no loss of generality in assuming that ${\mathcal C}$ is strict because by Theorem XI.5.3  of \cite{Christian} (this result implies the Mac Lane's coherence theorem) we know that every monoidal category is monoidally equivalent to a strict one. Then, we may work as if the constrains were all identities. We also assume that in ${\mathcal C}$ every idempotent morphism splits, i.e., for any morphism $q:M\rightarrow M$ such that $q\co q=q$ there exists an object $N$, called the image of $q$, and morphisms $i:N\rightarrow M$, $p:M\rightarrow N$ such that $q=i\co p$ and $p\co i=id_N$. The morphisms $p$ and $i$ will be called a factorization of $q$. Note that $N$, $p$ and $i$ are unique up to isomorphism. The categories satisfying this property constitute a broad
class that includes, among others, the categories with epi-monic decomposition for
morphisms and categories with (co)equalizers.  Finally, given
objects $A$, $B$, $D$ and a morphism $f:B\rightarrow D$, we write
$A\ot f$ for $id_{A}\ot f$ and $f\ot A$ for $f\ot id_{A}$.

An monoid in ${\mathcal C}$ is a triple $A=(A, \eta_{A},
\mu_{A})$ where $A$ is an object in ${\mathcal C}$ and
 $\eta_{A}:K\rightarrow A$ (unit), $\mu_{A}:A\ot A
\rightarrow A$ (product) are morphisms in ${\mathcal C}$ such that
$\mu_{A}\co (A\ot \eta_{A})=id_{A}=\mu_{A}\co (\eta_{A}\ot A)$,
$\mu_{A}\co (A\ot \mu_{A})=\mu_{A}\co (\mu_{A}\ot A)$. Given two
monoids $A= (A, \eta_{A}, \mu_{A})$ and $B=(B, \eta_{B}, \mu_{B})$,
$f:A\rightarrow B$ is a monoid morphism if $\mu_{B}\co (f\ot
f)=f\co \mu_{A}$, $ f\co \eta_{A}= \eta_{B}$.

A comonoid in ${\mathcal C}$ is a triple ${D} = (D,
\varepsilon_{D}, \delta_{D})$ where $D$ is an object in ${\mathcal
C}$ and $\varepsilon_{D}: D\rightarrow K$ (counit),
$\delta_{D}:D\rightarrow D\ot D$ (coproduct) are morphisms in
${\mathcal C}$ such that $(\varepsilon_{D}\ot D)\co \delta_{D}=
id_{D}=(D\ot \varepsilon_{D})\co \delta_{D}$, $(\delta_{D}\ot D)\co
\delta_{D}=
 (D\ot \delta_{D})\co \delta_{D}.$ If ${D} = (D, \varepsilon_{D},
 \delta_{D})$ and
${ E} = (E, \varepsilon_{E}, \delta_{E})$ are comonoids,
$f:D\rightarrow E$ is a comonoid morphism if $(f\ot f)\co
\delta_{D} =\delta_{E}\co f$, $\varepsilon_{E}\co f
=\varepsilon_{D}.$

 Let  $A$ be a monoid. The pair
$(M,\varphi_{M})$ is a left $A$-module if $M$ is an object in
${\mathcal C}$ and $\varphi_{M}:A\ot M\rightarrow M$ is a morphism
in ${\mathcal C}$ satisfying $\varphi_{M}\circ( \eta_{A}\ot
M)=id_{M}$, $\varphi_{M}\circ (A\ot \varphi_{M})=\varphi_{M}\circ
(\mu_{A}\ot M)$. Given two left ${A}$-modules $(M,\varphi_{M})$ and
$(N,\varphi_{N})$, $f:M\rightarrow N$ is a morphism of left
${A}$-modules if $\varphi_{N}\circ (A\ot f)=f\circ \varphi_{M}$. In
a similar way we can define the notions of right $A$-module and
morphism of right $A$-modules. In this case we denote the left
action by $\phi_{M}$.

\section{Equivalent weak crossed products}

In the first paragraphs of this section we resume some basic facts
about the general theory of weak crossed products. The complete details can be found in \cite{mra-preunit}.

Let $A$ be a monoid and $V$ be an object in
${\mathcal C}$. Suppose that there exists a morphism
$$\psi_{V}^{A}:V\ot A\rightarrow A\ot V$$  such that the following
equality holds
\begin{equation}\label{wmeas-wcp}
(\mu_A\ot V)\co (A\ot \psi_{V}^{A})\co (\psi_{V}^{A}\ot A) =
\psi_{V}^{A}\co (V\ot \mu_A).
\end{equation}
 As a consequence of (\ref{wmeas-wcp}), the morphism $\nabla_{A\ot V}:A\ot V\rightarrow
A\ot V$ defined by
\begin{equation}\label{idem-wcp}
\nabla_{A\ot V} = (\mu_A\ot V)\co(A\ot \psi_{V}^{A})\co (A\ot V\ot
\eta_A)
\end{equation}
is  idempotent. Moreover, $\nabla_{A\ot V}$ satisfies that
$$\nabla_{A\ot V}\co (\mu_A\ot V) = (\mu_A\ot V)\co
(A\ot \nabla_{A\ot V}),$$ that is, $\nabla_{A\ot V}$ is a left
$A$-module morphism (see Lemma 3.1 of \cite{mra-preunit}) for the
left  action  $\varphi_{A\ot V}=\mu_{A}\ot V$. With $A\times V$,
$i_{A\ot V}:A\times V\rightarrow A\ot V$ and $p_{A\ot V}:A\ot
V\rightarrow A\times V$ we denote the object, the injection and the
projection associated to the factorization of $\nabla_{A\ot V}$.
Finally, if $\psi_{V}^{A}$ satisfies (\ref{wmeas-wcp}), the following
identities hold
\begin{equation}\label{fi-nab}
(\mu_{A}\ot V)\co (A\ot \psi_{V}^{A})\co (\nabla_{A\ot V}\ot A)=
(\mu_{A}\ot V)\co (A\ot \psi_{V}^{A})=\nabla_{A\ot V}\co(\mu_{A}\ot
V)\co (A\ot \psi_{V}^{A}).
\end{equation}

From now on we consider quadruples ${\Bbb A}_{V}=(A, V,
\psi_{V}^{A}, \sigma_{V}^{A})$ where $A$ is a monoid, $V$ an
object, $\psi_{V}^{A}:V\ot A\rightarrow A\ot V$ a morphism
satisfiying (\ref{wmeas-wcp}) and $\sigma_{V}^{A}:V\ot V\rightarrow
A\ot V$  a morphism in ${\mathcal C}$.

We say that ${\Bbb A}_{V}=(A, V, \psi_{V}^{A}, \sigma_{V}^{A})$
satisfies the twisted condition if
\begin{equation}\label{twis-wcp}
(\mu_A\ot V)\co (A\ot \psi_{V}^{A})\co (\sigma_{V}^{A}\ot A) =
(\mu_A\ot V)\co (A\ot \sigma_{V}^{A})\co (\psi_{V}^{A}\ot V)\co
(V\ot \psi_{V}^{A})
\end{equation}
and   the  cocycle
condition holds if
\begin{equation}\label{cocy2-wcp}
(\mu_A\ot V)\co (A\ot \sigma_{V}^{A}) \co (\sigma_{V}^{A}\ot V) =
(\mu_A\ot V)\co (A\ot \sigma_{V}^{A})\co (\psi_{V}^{A}\ot V)\co
(V\ot\sigma_{V}^{A}).
\end{equation}

Note that, if ${\Bbb A}_{V}=(A, V, \psi_{V}^{A}, \sigma_{V}^{A})$
satisfies the twisted condition in Proposition 3.4 of
\cite{mra-preunit} we prove that the following equalities hold:
\begin{equation}\label{c1}
(\mu_A\otimes V)\circ (A\otimes \sigma_{V}^{A})\circ
(\psi_{V}^{A}\otimes V)\circ (V\otimes \nabla_{A\otimes V}) =
\nabla_{A\otimes V}\circ (\mu_A\otimes V)\circ (A\otimes
\sigma_{V}^{A})\circ (\psi_{V}^{A}\otimes V),
\end{equation}
\begin{equation}\label{aw}
\nabla_{A\otimes V}\circ (\mu_A\otimes V)\circ
(A\otimes\sigma_{V}^{A})\circ (\nabla_{A\otimes V}\otimes V) =
\nabla_{A\otimes V}\circ (\mu_A\otimes V)\circ
(A\otimes\sigma_{V}^{A}).
\end{equation}

Then, if $\nabla_{A\ot V}\co\sigma_{V}^{A}=\sigma_{V}^{A}$ we obtain
\begin{equation}\label{c11}
(\mu_A\otimes V)\circ (A\otimes \sigma_{V}^{A})\circ
(\psi_{V}^{A}\otimes V)\circ (V\otimes \nabla_{A\otimes V}) =
 (\mu_A\otimes V)\circ (A\otimes
\sigma_{V}^{A})\circ (\psi_{V}^{A}\otimes V),
\end{equation}
\begin{equation}\label{aw1}
 (\mu_A\otimes V)\circ
(A\otimes\sigma_{V}^{A})\circ (\nabla_{A\otimes V}\otimes V) =
(\mu_A\otimes V)\circ (A\otimes\sigma_{V}^{A}).
\end{equation}

By virtue of (\ref{twis-wcp}) and (\ref{cocy2-wcp}) we will consider
from now on, and without loss of generality, that
\begin{equation}
\label{idemp-sigma-inv} \nabla_{A\ot V}\co\sigma_{V}^{A} =
\sigma_{V}^{A}
\end{equation}
holds for all quadruples ${\Bbb A}_{V}=(A, V, \psi_{V}^{A},
\sigma_{V}^{A})$ {(see Proposition 3.7 of \cite{mra-preunit})}.

For ${\Bbb A}_{V}=(A, V, \psi_{V}^{A}, \sigma_{V}^{A})$ define the
product
\begin{equation}\label{prod-todo-wcp}
\mu_{A\ot  V} = (\mu_A\ot V)\co (\mu_A\ot \sigma_{V}^{A})\co (A\ot
\psi_{V}^{A}\ot V)
\end{equation}
and let $\mu_{A\times V}$ be the product
\begin{equation}
\label{prod-wcp} \mu_{A\times V} = p_{A\ot V}\co\mu_{A\ot V}\co
(i_{A\ot V}\ot i_{A\ot V}).
\end{equation}

If the twisted and the cocycle conditions hold, the product
$\mu_{A\ot V}$ is associative and normalized with respect to
$\nabla_{A\ot V}$, i.e.,  
\begin{equation}
\label{normalized}
\nabla_{A\ot V}\co \mu_{A\ot V}=\mu_{A\ot
V}=\mu_{A\ot V}\co (\nabla_{A\ot V}\ot \nabla_{A\ot V}),
\end{equation}
and by the
definition of $\mu_{A\ot V}$ we have
\begin{equation}
\label{otra-prop} \mu_{A\ot V}\co (\nabla_{A\ot V}\ot A\ot
V)=\mu_{A\ot V}
\end{equation}
and therefore
\begin{equation}
\label{vieja-proof} \mu_{A\otimes V}\circ (A\otimes V\otimes
\nabla_{A\otimes V})=\mu_{A\otimes V}.
\end{equation}
 Due to the normality condition, $\mu_{A\times V}$ is
associative as well (Propostion 3.8 of \cite{mra-preunit}). Hence we
define:

\begin{defin}\label{wcp-def}{\rm
If ${\Bbb A}_{V}=(A, V, \psi_{V}^{A}, \sigma_{V}^{A})$  satisfies
(\ref{twis-wcp}) and (\ref{cocy2-wcp}) we say that $(A\ot V,
\mu_{A\ot V})$ is a weak crossed product.

Trivially, $\mu_{A\ot V}$ is left $A$-linear  for the left  actions $\varphi_{A\otimes
V}$, and $\varphi_{A\otimes V\otimes A\otimes V
}=\varphi_{A\otimes V}\otimes  A\otimes V$. Moreover, the restricted product $\mu_{A\times V}$ is left $A$-linear for $\varphi_{A\times V}=p_{A\ot V}\co \varphi_{A\otimes V}\co (A\ot i_{A\ot V})$ and $\varphi_{A\times V\ot A\times V}=
\varphi_{A\times V} \ot A\times V$.
}
\end{defin}

The next natural question that arises is if it is possible to endow
$A\times V$ with a unit, and hence with a monoid structure. As
$A\times V$ is given as an image of an idempotent, it seems
reasonable to use the notion of  preunit 
to obtain an unit. In our setting, if $A$ is a monoid, $V$ an
object in ${\mathcal C}$ and $m_{A\otimes V}$ is an associative
product defined in $A\otimes V$ a preunit $\nu_{V}:K\rightarrow A\otimes
V$ is a morphism satisfying
\begin{equation}
m_{A\otimes V}\circ (A\otimes V\otimes \nu_{V})=m_{A\otimes V}\circ
(\nu_{V}\otimes A\otimes V)=m_{A\otimes V}\circ (A\otimes V\otimes
(m_{A\otimes V}\circ (\nu_{V}\otimes \nu_{V}))).
\end{equation}
Associated to a preunit we obtain an idempotent morphism
$$\nabla_{A\otimes
V}^{\nu_{V}}=m_{A\otimes V}\circ (A\otimes V\otimes \nu_{V}):A\otimes
V\rightarrow A\otimes V.$$ Take $A\times V$ the image of this
idempotent, $p_{A\otimes V}^{\nu_{V}}$ the projection and $i_{A\otimes
V}^{\nu_{V}}$ the injection. It is possible to endow $A\times V$ with a
monoid structure whose product is $$m_{A\times V} = p_{A\otimes
V}^{\nu_{V}}\circ m_{A\otimes V}\circ (i_{A\otimes V}^{\nu_{V}}\otimes
i_{A\otimes V}^{\nu_{V}})$$ and whose unit is $\eta_{A\times
V}=p_{A\otimes V}^{\nu_{V}}\circ \nu_{V}$ (see Proposition 2.5 of
\cite{mra-preunit}). If moreover, $m_{A\otimes V}$ is left
$A$-linear for the actions $\varphi_{A\otimes V}$,
$\varphi_{A\otimes V\otimes A\otimes V }$ and normalized with respect to $\nabla_{A\otimes
V}^{\nu_{V}}$,  the morphism
\begin{equation}
\label{beta-nu} \beta_{\nu_{V}}:A\rightarrow A\otimes V,\; \beta_{\nu_{V}} =
(\mu_A\otimes V)\circ (A\otimes \nu_{V})
\end{equation}
is multiplicative and left $A$-linear for $\varphi_{A}=\mu_{A}$.

Although $\beta_{\nu_{V}}$ is not a monoid morphism, because $A\otimes
V$ is not a monoid, we have that $\beta_{\nu_{V}}\circ \eta_A = \nu_{V}$,
and thus the morphism $\bar{\beta_{\nu_{V}}} = p_{A\otimes
V}^{\nu_{V}}\circ\beta_{\nu_{V}}:A\rightarrow A\times V$ is a monoid
morphism.

In light of the considerations made in the last paragraphs, and
using the twisted and the cocycle conditions, in \cite{mra-preunit}
we characterize weak crossed products with a preunit, and moreover
we obtain a monoid structure on $A\times V$. These assertions are
a consequence of the following results proved in \cite{mra-preunit}.

\begin{teo}
\label{thm1-wcp} Let $A$ be a monoid, $V$ an object and
$m_{A\otimes V}:A\otimes V\otimes A\otimes V\rightarrow A\otimes V$
a morphism of left $A$-modules  for the left  action $\varphi_{A\otimes
V}$ and  $\varphi_{A\otimes V\otimes A\otimes V}$.

Then the following statements are equivalent:
\begin{itemize}
\item[(i)] The product $m_{A\otimes V}$ is associative with preunit
$\nu_{V}$ and normalized with respect to $\nabla_{A\otimes V}^{\nu_{V}}.$

\item[(ii)] There exist morphisms $\psi_{V}^{A}:V\otimes A\rightarrow A\otimes
V$, $\sigma_{V}^{A}:V\otimes V\rightarrow A\otimes V$ and
$\nu_{V}:k\rightarrow A\otimes V$ such that if $\mu_{A\otimes V}$ is the
product defined in (\ref{prod-todo-wcp}), the pair $(A\otimes V,
\mu_{A\otimes V})$ is a weak crossed product with $m_{A\otimes V} =
\mu_{A\otimes V}$ satisfying:
    \begin{equation}\label{pre1-wcp}
    (\mu_A\otimes V)\circ (A\otimes \sigma_{V}^{A})\circ
    (\psi_{V}^{A}\otimes V)\circ (V\otimes \nu_{V}) =
    \nabla_{A\otimes V}\circ
    (\eta_A\otimes V),
    \end{equation}
    \begin{equation}\label{pre2-wcp}
    (\mu_A\otimes V)\circ (A\otimes \sigma_{V}^{A})\circ
    (\nu_{V}\otimes V) = \nabla_{A\otimes V}\circ (\eta_A\otimes V),
    \end{equation}
    \begin{equation}\label{pre3-wcp}
(\mu_A\otimes V)\circ (A\otimes \psi_{V}^{A})\circ (\nu_{V}\otimes A) =
\beta_{\nu_{V}},
\end{equation}
\end{itemize}
where $\beta_{\nu_{V}}$ is the morphism defined in (\ref{beta-nu}). In
this case $\nu_{V}$ is a preunit for $\mu_{A\otimes V}$, the idempotent
morphism of the weak crossed product $\nabla_{A\otimes V}$ is the
idempotent $\nabla_{A\otimes V}^{\nu_{V}}$.
\end{teo}

\begin{nota}
\label{proof-resume} {\rm Note that in the proof of the previous
theorem, we obtain that  
\begin{equation}\label{fi-wcp}
\psi_{V}^{A} = \mu_{A\otimes V}\circ (\eta_A\otimes V\otimes
\beta_{\nu_{V}}),
\end{equation}
\begin{equation}\label{sigma-wcp}
\sigma_{V}^{A} = \mu_{A\otimes V}\circ (\eta_A\otimes V\otimes
\eta_A\otimes V), 
\end{equation}
hold. Also, by (\ref{pre3-wcp}), we have 
\begin{equation}\label{preunit-idemp}
\nabla_{A\ot V}\co \nu_{V}=\nu_{V}.
\end{equation}
}
\end{nota}

\begin{defin}\label{wcp-pre-def}
{\rm We will say that a weak crossed product $(A\otimes V, \mu_{A\otimes V})$ is a weak crossed product with preunit 
$\nu_{V}:k\rightarrow A\otimes V$, if  (\ref{pre1-wcp}), (\ref{pre2-wcp}) and (\ref{pre3-wcp}) hold.}
\end{defin}

Then, as a corollary of Proposition 2.5 of \cite{mra-preunit} and Theorem \ref{thm1-wcp}, we have the following corollary.

\begin{corol}\label{corol-wcp}
If $(A\otimes V, \mu_{A\otimes V})$ is a weak crossed product with
preunit $\nu_{V}$, then $A\times V$ is a monoid with the product
defined in (\ref{prod-wcp}) and unit $\eta_{A\times V}=p_{A\otimes
V}\circ\nu_{V}$.
\end{corol}

In the following definition we introduce the notion of equivalent weak crossed products. 

\begin{defin}\label{wcp-equiv-def}
{\rm Let  $(A\otimes V, \mu_{A\otimes V})$ and $(A\otimes W, \mu_{A\otimes W})$ be weak crossed products  with preunits  $\nu_{V}$ and $\nu_{W}$ respectively.  We will say that $(A\otimes V, \mu_{A\otimes V})$ and $(A\otimes W, \mu_{A\otimes W})$ are equivalent if there exists a monoid isomorphism $\alpha:A\times V\rightarrow A\times W$ of left $A$-modules for the actions $\varphi_{A\times V}$ and $\varphi_{A\times W}$.}
\end{defin}

The  notion of equivalence for weak crossed products is characterized by the main theorem of this section:

\begin{teo}
\label{thm2-equiv}
Let $(A\otimes V, \mu_{A\otimes V})$ and $(A\otimes W, \mu_{A\otimes W})$ be weak crossed products  with preunits  $\nu_{V}$ and $\nu_{W}$ respectively.  The following assertions are equivalent:

\item[(i)] The weak crossed products  $(A\otimes V, \mu_{A\otimes V})$ and $(A\otimes W, \mu_{A\otimes W})$ are equivalent.

\item[(ii)] There exist two morphisms 
$$T:A\ot V\rightarrow A\ot W,\;\; S:A\ot W\rightarrow A\ot V$$
of left $A$-modules for the actions $\varphi_{A\otimes V}$, $\varphi_{A\otimes W}$ satisfying the conditions 
\begin{equation}
\label{preserv-preunit}
T\co \nu_{V}=\nu_{W},
\end{equation}
\begin{equation}
\label{preserv-product}
T\co \mu_{A\otimes V}=\mu_{A\otimes W}\co (T\ot T),
\end{equation}
\begin{equation}
\label{preserv-idemp}
S\co T=\nabla_{A\ot V},\;\; T\co S=\nabla_{A\ot W},
\end{equation}

\item[(iii)] There exist two morphisms 
$$\gamma:V\rightarrow A\ot W,\;\; \theta:W\rightarrow A\ot V$$
satisfying the conditions 
\begin{equation}
\label{gamma-theta-preunit}
\nu_{V}=(\mu_{A}\ot V)\co (A\ot \theta)\co \nu_{W},
\end{equation}
\begin{equation}
\label{gamma-theta-idemp}
\theta=\nabla_{A\ot V}\co \theta,
\end{equation}
\begin{equation}
\label{gamma-theta-psi}
\psi_{W}^{A}=(\mu_{A}\ot W)\co (\mu_{A}\ot \gamma)\co (A\ot \psi_{V}^{A})\co (\theta\ot A),
\end{equation}
\begin{equation}
\label{gamma-theta-sigma}
\sigma_{W}^{A}=(\mu_{A}\ot W)\co (A\ot \gamma)\co \mu_{A\ot V}\co (\theta\ot \theta) ,
\end{equation}
\begin{equation}
\label{gamma-theta-special}
(\mu_{A}\ot V)\co (A\ot \theta)\co \gamma=\nabla_{A\ot V}\co (\eta_{A}\ot V).
\end{equation}
\end{teo}

\begin{proof} We begin by proving (i)$\Rightarrow $ (ii). Let $\alpha:A\times V\rightarrow A\times W$ be the monoid isomorphism of left $A$-modules for the actions $\varphi_{A\times V}$ and $\varphi_{A\times W}$. Define 
$$T=i_{A\ot W}\co \alpha \co p_{A\ot V}, \;\; S=i_{A\ot V}\co \alpha^{-1} \co p_{A\ot W}.$$
Then, (\ref{preserv-idemp}) and 
 \begin{equation}
\label{preserv-comp}
T\co S\co T=T,\;\; S\co T\co S=S.
\end{equation}
hold trivially. Also, using that $\alpha$ and $\alpha^{-1}$ are  monoid morphisms and 
(\ref{preunit-idemp}) for $\nu_{V}$ and $\nu_{W}$, we obtain (\ref{preserv-preunit}). On the other hand, $T$ is a morphism of left $A$-modules because, using that $\nabla_{A\ot V}$ is a morphism of left $A$-modules, we have,
$$p_{A\ot W}\co \varphi_{A\otimes W} \co (A\ot T)=\varphi_{A\times W}\co (A\ot \alpha)\co (A\ot p_{A\ot V})=
\alpha\co \varphi_{A\times V}\co (A\ot p_{A\ot V})$$
$$=\alpha\co p_{A\ot V}\co (\mu_{A}\ot V)\co (A\ot \nabla_{A\ot V})=\alpha\co p_{A\ot V}\co (\mu_{A}\ot V)=\alpha\co p_{A\ot V}\co \varphi_{A\otimes V}$$
and, as a consequence, composing with $i_{A\ot W}$, applying that  $\nabla_{A\ot W}$ is a morphism of left $A$-modules,
and by (\ref{preserv-comp}), 
$$T\co \varphi_{A\otimes V}= \varphi_{A\otimes W} \co (A\ot T)=\varphi_{A\otimes W} \co (A\ot (\nabla_{A\ot W}\co T))
=\varphi_{A\otimes W} \co (A\ot (T\co S\co T))=\varphi_{A\otimes W} \co (A\ot  T).$$
Similarly, we can prove that $S$ is a morphism of left $A$-modules.  Finally, we will obtain (\ref{preserv-product}), i.e, $T$ is multiplicative. Indeed: composing with $p_{A\ot W}$ we have 
\begin{itemize}

\item[ ]$\hspace{0.38cm}p_{A\ot W}\co  \mu_{A\ot W}\co (T\ot T) $

\item[ ]$= \mu_{A\times W}\co ((\alpha\co p_{A\ot V})\ot (\alpha\co p_{A\ot V}))$

\item[ ]$=\alpha\co  \mu_{A\times V}\co (p_{A\ot V}\ot p_{A\ot V})$

\item[ ]$= \alpha\co p_{A\ot V}\co  \mu_{A\otimes V}\co (\nabla_{A\ot V}\ot \nabla_{A\ot V}) $

\item[ ]$=p_{A\ot W} \co T\co \mu_{A\otimes V},  $

\end{itemize}

where the first equality follows by the definition of $\mu_{A\times W}$, in the second one we used that $\alpha$ is a monoid morphism, the third one relies on the definition of $\mu_{A\times V}$, and the last one follows by (\ref{normalized}) for $\mu_{A\ot V}$, and the properties of the projection and the injection associated to $\nabla_{A\ot W}$. 

Therefore, by (\ref{preserv-comp}) and (\ref{normalized}) for $\mu_{A\ot W}$, we have
$$T\co \mu_{A\otimes V}=T\co S\co T\co \mu_{A\otimes V}=\nabla_{A\ot W} \co T\co \mu_{A\otimes V}=\nabla_{A\ot W}\co  \mu_{A\ot W}\co (T\ot T)
=\mu_{A\ot W}\co (T\ot T).$$

The proof for (ii)$\Rightarrow $ (i) is the following. First note that, by (\ref{preserv-idemp}), $p_{A\ot V}=p_{A\ot V}\co \nabla_{A\ot V}=p_{A\ot V}\co S\co T$ and then 
\begin{equation}
\label{eq1}
p_{A\ot V}\co S=p_{A\ot V}\co S\co T\co S=p_{A\ot V}\co S\co \nabla_{A\ot W}.
\end{equation}
Similarly, 
\begin{equation}
\label{eq2}
p_{A\ot W}\co T=p_{A\ot W}\co T\co S\co T=p_{A\ot V}\co T\co \nabla_{A\ot V},
\end{equation}
\begin{equation}
\label{eq3}
S\co i_{A\ot W}= S\co T\co S\co i_{A\ot W} =\nabla_{A\ot V}\co S\co i_{A\ot W},
\end{equation}
and
\begin{equation}
\label{eq4}
T\co i_{A\ot V}= T\co S\co T\co i_{A\ot V} =\nabla_{A\ot W}\co T\co i_{A\ot V},.
\end{equation}
hold.

Set 
$$\alpha=p_{A\ot W}\co T \co i_{A\ot V}.$$
The, $\alpha$ is an isomorphism with inverse $\alpha^{\prime}=p_{A\ot V}\co S \co i_{A\ot W}.$ Indeed:  by (\ref{eq1}) and (\ref{preserv-idemp}),  we have 
$$\alpha^{\prime}\co \alpha=p_{A\ot V}\co S\co \nabla_{A\ot W}\co T \co i_{A\ot V}=p_{A\ot V}\co S\co  T \co i_{A\ot V}=p_{A\ot V}\co \nabla_{A\ot V}\co i_{A\ot V}=id_{A\times V},$$ 
and, similarly, $\alpha\co \alpha^{\prime}=id_{A\times W}$. Also,  by (\ref{eq2}) and (\ref{preserv-preunit})
$$\alpha \co \eta_{A\times V}=p_{A\ot W} \co  T\co \nabla_{A\ot V}\co \nu_{V}=p_{A\ot W} \co  T\co  \nu_{V} =p_{A\ot W} \co \nu_{W} = \eta_{A\times W}.$$
On the other hand, $\alpha$ is multiplicative because, by (\ref{normalized}) for $\mu_{A\times V}$ and $\mu_{A\ot W}$, and (\ref{preserv-product}) we have 

\begin{itemize}

\item[ ]$\hspace{0.38cm} \mu_{A\times W}\co (\alpha\ot \alpha) $

\item[ ]$= p_{A\ot W}\co \mu_{A\otimes W}\co ((\nabla_{A\ot W}\co T\co i_{A\ot V})\ot (\nabla_{A\ot W}\co T\co i_{A\ot V}))$

\item[ ]$=p_{A\ot W}\co \mu_{A\otimes W}\co ((T\co i_{A\ot V})\ot (T\co i_{A\ot V}))$

\item[ ]$= p_{A\ot W}\co T\co \mu_{A\otimes V}\co ( i_{A\ot V}\ot  i_{A\ot V}) $

\item[ ]$=p_{A\ot W}\co T\co \nabla_{A\ot V}\co\mu_{A\otimes V}\co ( i_{A\ot V}\ot  i_{A\ot V})   $

\item[ ]$=\alpha\co \mu_{A\times V}$, 

\end{itemize}

and finally, using that $\nabla_{A\ot W}$ and $T$ are morphism of left $A$-modules, as well as  (\ref{eq4}) and (\ref{eq2}),  we obtain 
$$\varphi_{A\times W}\co (A\ot \alpha)=p_{A\ot W}\co 
(\mu_{A}\ot W)\co (A\ot (\nabla_{A\ot W}\co T\co i_{A\ot V}))=p_{A\ot W}\co  \varphi_{A\ot W}\co (A\ot ( T\co i_{A\ot V}))$$
$$=p_{A\ot W}\co T\co \varphi_{A\ot V}\co (A\ot i_{A\ot V})=p_{A\ot W}\co T\co \nabla_{A\ot V}\co \varphi_{A\ot V}\co (A\ot i_{A\ot V})=\alpha\co \varphi_{A\times V}, $$
and $\alpha$ is a morphism of left $A$-modules.

Note that, if (\ref{preserv-product}) and (\ref{preserv-idemp}) hold,  we also have that $S$ is multiplicative, i.e., 
\begin{equation}
\label{preserv-product-s}
S\co \mu_{A\otimes W}=\mu_{A\otimes V}\co (S\ot S),
\end{equation}
Indeed: by (\ref{normalized}) for $\mu_{A\ot V}$ and $\mu_{A\ot W}$, 
$$\mu_{A\ot V}\co (S\ot S)=\nabla_{A\ot V}\co \mu_{A\ot V}\co (S\ot S)=S\co T\co \mu_{A\ot V}\co (S\ot S)
=S\co \mu_{A\ot W}\co ((T\co S)\ot (T\co S)) $$
$$=S\co \mu_{A\ot W}\co   (\nabla_{A\ot W}\ot \nabla_{A\ot W})=
S\co \mu_{A\ot W}.$$
Moreover, if (\ref{preserv-idemp}) holds $T$ is multiplicative if, and only if, $S$ is multiplicative.

In the next step of the proof we will prove that (ii) $\Rightarrow $ (iii).  First note that if $T$ and $S$ satisfy 
(\ref{preserv-preunit}) and (\ref{preserv-idemp}), the following identity holds 
\begin{equation}
\label{preserv-preunit-b}
S\co \nu_{W}=\nu_{V},
\end{equation}
because $\nu_{V}=\nabla_{A\ot V}\co \nu_{V}=S\co T\co \nu_{V}=S\co \nu_{W}.$

Consider 
$$\gamma=T\co (\eta_{A}\ot V),\;\; \theta=\nabla_{A\ot V}\co S\co (\eta_{A}\ot W).$$
Then, using that $S$ is a morphism of left $A$-modules, (\ref{preserv-preunit-b}), and (\ref{preunit-idemp}), we prove (\ref{gamma-theta-preunit}):
$$(\mu_{A}\ot V)\co (A\ot \theta)\co \nu_{W}=(\mu_{A}\ot V)\co (A\ot (\nabla_{A\ot V}\co S\co (\eta_{A}\ot W)))\co \nu_{W}
=\nabla_{A\ot V}\co S\co \nu_{W}=\nabla_{A\ot V}\co\nu_{V}=\nu_{V}.$$
Moreover, (\ref{gamma-theta-idemp})  follows trivially because $\nabla_{A\ot V}$ is idempotent. 

The proof for (\ref{gamma-theta-psi}) is the following:
\begin{itemize}

\item[ ]$\hspace{0.38cm} (\mu_{A}\ot W)\co (\mu_{A}\ot \gamma)\co (A\ot \psi_{V}^{A})\co (\theta\ot A) $

\item[ ]$=(\mu_{A}\ot W)\co (A\ot (T\co \psi_{V}^{A}))\co (\theta\ot A) $

\item[ ]$=(\mu_{A}\ot W)\co (A\ot (T\co (\mu_{A\otimes V}\circ (\eta_A\otimes V\otimes
\beta_{\nu_{V}}))))\co (\theta\ot A)$

\item[ ]$=(\mu_{A}\ot W)\co (A\ot (\mu_{A\otimes W}\circ ((T\co (\eta_A\otimes V))\otimes
(T\co\beta_{\nu_{V}}))))\co (\theta\ot A)  $

\item[ ]$= (\mu_{A}\ot W)\co (A\ot (\mu_{A\otimes W}\circ ((T\co (\eta_A\otimes V))\otimes
((\mu_{A}\ot V)\co (A\ot (T\co \nu_{V}))))))\co (\theta\ot A) $

\item[ ]$= \mu_{A\otimes W}\circ  (((\mu_{A}\ot W)\co (A\ot (T\co (\eta_{A}\ot V))))\ot \beta_{\nu_{W}}) \co (\theta\ot A)$

\item[ ]$= \mu_{A\otimes W}\circ ((T\co \nabla_{A\ot V}\co  S\co (\eta_{A}\ot W))\ot \beta_{\nu_{W}}) $

\item[ ]$= \mu_{A\otimes W}\circ ((T\co S\co T\co  S\co (\eta_{A}\ot W))\ot \beta_{\nu_{W}})$

\item[ ]$= \mu_{A\otimes W}\circ ((\nabla_{A\ot W}\co \nabla_{A\ot W}\co (\eta_{A}\ot W))\ot \beta_{\nu_{W}})$

\item[ ]$= \mu_{A\otimes W}\circ ((\nabla_{A\ot W}\co (\eta_{A}\ot W))\ot \beta_{\nu_{W}})$

\item[ ]$=\mu_{A\otimes W}\circ (\eta_A\otimes W\otimes
\beta_{\nu_{W}}) $

\item[ ]$=\psi_{W}^{A},$

\end{itemize}

where the first, the fourth and the sixth equalities follow by the left linearity of $T$, the second one follows by (\ref{fi-wcp}) for $\psi_{V}^{A}$, and the third one follows because $T$ is multiplicative. In the fifth equality we used that $\mu_{A\otimes W}$ is left linear and (\ref{preserv-preunit}). The seventh and eighth ones are a consequence of   (\ref{preserv-idemp}), the ninth one follows by the idempotent character of $\nabla_{A\ot W}$ and in the tenth one we used the normality condition for the product $\mu_{A\ot W}$. Finally, 
 the last one follows by (\ref{fi-wcp}) for $\psi_{W}^{A}$.

On the other hand:
\begin{itemize}

\item[ ]$\hspace{0.38cm} (\mu_{A}\ot W)\co (A\ot \gamma)\co \mu_{A\ot V}\co (\theta\ot \theta) $

\item[ ]$=  (\mu_{A}\ot W)\co (A\ot \gamma)\co \mu_{A\ot V}\co ((\nabla_{A\ot V}\co S\co (\eta_{A}\ot W))\ot (\nabla_{A\ot V}\co S\co (\eta_{A}\ot W)))$

\item[ ]$=(\mu_{A}\ot W)\co (A\ot (T\co (\eta_{A}\ot V)))\co S\co \mu_{A\ot W}\co (\eta_{A}\ot W\ot \eta_{A}\ot W) $

\item[ ]$= \nabla_{A\ot W}\co \sigma_{W}^{A} $
\item[ ]$=\sigma_{W}^{A},  $

\end{itemize}
and then we proved (\ref{gamma-theta-sigma}). In the previous identities, the first one follows by definition, the second one follows by the normality condition for the product $\mu_{A\ot W}$ as well as because $S$ is multiplicative, the third one relies on (\ref{sigma-wcp}) and in the left linearity of $T$, and, finally, the fourth one  is a consequence of the properties of $\sigma_{W}^{A}$.

Finally, (\ref{gamma-theta-special}) follows easily using that $\nabla_{A\ot W}$ and $S$ are left $A$-linear, 
(\ref{preserv-idemp}), and the idempotent character of $\nabla_{A\ot V}$.

In the last step of  this proof we prove that (iii) $\Rightarrow $ (ii). First, note that if (\ref{gamma-theta-psi}) and (\ref{gamma-theta-idemp}) hold, then 
\begin{equation}
\label{gamma-theta-special-2}
(\mu_{A}\ot W)\co (A\ot \gamma)\co \theta=\nabla_{A\ot W}\co (\eta_{A}\ot W).
\end{equation}
also holds. Indeed:
$$\nabla_{A\ot W}\co (\eta_{A}\ot W)=\psi_{W}^{A}\co (W\ot \eta_{A})=(\mu_{A}\ot W)\co (\mu_{A}\ot \gamma)\co (A\ot \psi_{V}^{A})\co (\theta\ot \eta_{A})$$
$$= (\mu_{A}\ot W)\co (A\ot \gamma)\co\nabla_{A\ot V}\co \theta=(\mu_{A}\ot W)\co (A\ot \gamma)\co \theta.$$

Then, as a consequence, we obtain 
\begin{equation}
\label{gamma-theta-preunit-b}
\nu_{W}=(\mu_{A}\ot W)\co (A\ot \gamma)\co \nu_{V},
\end{equation}
because 
$$(\mu_{A}\ot W)\co (A\ot \gamma)\co \nu_{V} =(\mu_{A}\ot W)\co   (A\ot \gamma)\co (\mu_{A}\ot V)\co (A\ot \theta)\co \nu_{W}$$
$$=(\mu_{A}\ot W)\co (A\ot ((\mu_{A}\ot W)\co  (A\ot \gamma)\co \theta))\co \nu_{W}=(\mu_{A}\ot W)\co (A\ot (\nabla_{A\ot W}\co (\eta_{A}\ot W))\co \nu_{W}=\nabla_{A\ot W}\co \nu_{W}= \nu_{W}.$$

Define
$$T=(\mu_{A}\ot W)\co (A\ot \gamma),\;\; S=(\mu_{A}\ot V)\co (A\ot \theta).$$
Then the morphisms $T$ and $S$ are left $A$-linear. Also,  by (\ref{gamma-theta-preunit-b}) we obtain  (\ref{preserv-preunit}). Moreover, by the associativity of $\mu_{A}$ and (\ref{gamma-theta-special}) we prove that $S\co T=\nabla_{A\ot V}$, and, using (\ref{gamma-theta-special-2}), we show that $T\co S=\nabla_{A\ot W}$ holds. 
Finally we prove that $T$ is multiplicative because 
\begin{itemize}

\item[ ]$\hspace{0.38cm}  \mu_{A\otimes W}\co (T\ot T)$

\item[ ]$=(\mu_{A}\ot W)\co (\mu_{A}\ot \sigma_{W}^{A})\co (A\ot \psi_{W}^{A}\ot W)\co (T\ot T) $

\item[ ]$=(\mu_{A}\ot W)\co (\mu_{A}\ot ((\mu_{A}\ot W)\co (A\ot \gamma)\co \mu_{A\ot V}\co (\theta\ot \theta)))\co (A\ot ((\mu_{A}\ot W)\co (\mu_{A}\ot \gamma)\co (A\ot \psi_{V}^{A})  $
\item[ ]$\hspace{0.38cm}\co (\theta\ot A))\ot W)\co (T\ot T)$

\item[ ]$= T\co (\mu_{A}\ot V)\co (\mu_{A}\ot (\mu_{A\ot V}\co (((\mu_{A}\ot V)\co (A\ot \theta)\co \gamma)\ot A\ot V)\co (\mu_{A}\ot \psi_{V}^{A}\ot \theta)\co (A\ot ((\mu_{A}\ot V)$
\item[ ]$\hspace{0.38cm}\co (A\ot \theta)\co \gamma)\ot T) $

\item[ ]$= T\co (\mu_{A}\ot V)\co (\mu_{A}\ot (\mu_{A\ot V}\co ((\nabla_{A\ot V}\co (\eta_{A}\ot V))\ot A\ot V)\co (\mu_{A}\ot \psi_{V}^{A}\ot \theta)\co (A\ot (\nabla_{A\ot V}\co (\eta_{A}\ot V))\ot T)  $

\item[ ]$= T\co \mu_{A\ot V}\co (\mu_{A}\ot V\ot \theta)\co (A\ot \psi_{V}^{A}\ot W)\co (A\ot V\ot T)$

\item[ ]$= T\co  (\mu_{A}\ot V)\co (\mu_{A}\ot \sigma_{V}^{A})\co (A\ot ((\mu_A\ot V)\co (A\ot \psi_{V}^{A})\co (\psi_{V}^{A}\ot A))\ot V)\co (A\ot V\ot A\ot \theta)\co (A\ot V\ot T)$

\item[ ]$=T\co  (\mu_{A}\ot V)\co (\mu_{A}\ot \sigma_{V}^{A})\co (A\ot  \psi_{V}^{A}\ot V)\co (A\ot V\ot ((\mu_{A}\ot V)\co (A\ot ((\mu_{A}\ot V)\co (A\ot \theta)\co \gamma))))$

\item[ ]$=T\co  (\mu_{A}\ot V)\co (\mu_{A}\ot \sigma_{V}^{A})\co (A\ot  \psi_{V}^{A}\ot V)\co (A\ot V\ot ((\mu_{A}\ot V)\co (A\ot (\nabla_{A\ot V}\co (\eta_{A}\ot V)))))  $

\item[ ]$= T\co \mu_{A\ot V}, $
\end{itemize}
where the first equality follows by definition, the second one follows by (\ref{gamma-theta-psi}) and (\ref{gamma-theta-sigma}), and the third one is a consequence of the associativity of $\mu_{A}$ as well as the left linearity of $\mu_{A\ot V}$. The fourth and the eighth ones rely on (\ref{gamma-theta-special}) and the fifth one follows by the left linearity of the morphisms $\mu_{A\ot V}$ and $\nabla_{A\ot V}$. The sixth one relies on the definition of $\mu_{A\ot V}$ and on the left linearity of $\mu_{A\ot V}$. Finally, the seventh one follows by the associativity of $\mu_{A}$ and (\ref{wmeas-wcp}) for $\psi_{V}^{A}$, and the last one is a consequence of (\ref{vieja-proof}) and  the  left linearity of $\nabla_{A\ot V}$.

\end{proof}

\begin{prop}
\label{prop-equiv}
Let $(A\otimes V, \mu_{A\otimes V})$ and $(A\otimes W, \mu_{A\otimes W})$ be weak crossed products. Assume that there exist two morphisms $\gamma:V\rightarrow A\ot W,\;$ $\theta:W\rightarrow A\ot V$
satisfying the conditions (\ref{gamma-theta-sigma}) and (\ref{gamma-theta-special}). Then the following equality holds:
\begin{equation}
\label{sup-11}
(\mu_{A}\ot W)\co (\mu_{A}\ot \sigma_{W}^{A})\co (A\ot \gamma\ot W)\co (\psi_{V}^{A}\ot W)\co (V\ot \gamma)=
(\mu_{A}\ot W)\co (A\ot \gamma)\co \sigma_{V}^{A}.
\end{equation}
\end{prop}

\begin{proof}  Indeed: 
\begin{itemize}

\item[ ]$\hspace{0.38cm} (\mu_{A}\ot W)\co (\mu_{A}\ot \sigma_{W}^{A})\co (A\ot \gamma\ot W)\co (\psi_{V}^{A}\ot W)\co (V\ot \gamma)$

\item[ ]$=(\mu_{A}\ot W)\co (\mu_{A}\ot ((\mu_{A}\ot W)\co (A\ot \gamma)\co \mu_{A\ot V}\co (\theta\ot \theta)))\co (A\ot \gamma\ot W)\co (\psi_{V}^{A}\ot W)\co (V\ot \gamma)  $

\item[ ]$=(\mu_{A}\ot W)\co (\mu_{A}\ot \gamma)\co (A\ot ( \mu_{A\ot V}\co (((\mu_{A}\ot V)\co (A\ot \theta)\co \gamma)\ot \theta)))\co (\psi_{V}^{A}\ot W)\co (V\ot \gamma) $

\item[ ]$= (\mu_{A}\ot W)\co (\mu_{A}\ot \gamma)\co (A\ot ( \mu_{A\ot V}\co ((\nabla_{A\ot V}\co (\eta_{A}\ot V))\ot \theta)))\co (\psi_{V}^{A}\ot W)\co (V\ot \gamma) $

\item[ ]$= (\mu_{A}\ot W)\co (A\ot \gamma)\co \mu_{A\ot V}\co   (\psi_{V}^{A}\ot \theta)\co (V\ot \gamma)$

\item[ ]$=   (\mu_{A}\ot W)\co (A\ot \gamma)\co (\mu_A\ot V)\co (\mu_A\ot \sigma_{V}^{A})\co (A\ot
\psi_{V}^{A}\ot V)\co   (\psi_{V}^{A}\ot \theta)\co (V\ot \gamma) $

\item[ ]$=  (\mu_{A}\ot W)\co (A\ot \gamma)\co (\mu_A\ot V)\co (A\ot \sigma_{V}^{A})\co  (\psi_{V}^{A}\ot V)\co 
(V\ot ((\mu_{A}\ot V)\co (A\ot \theta)\co \gamma))$

\item[ ]$=(\mu_{A}\ot W)\co (A\ot \gamma)\co (\mu_A\ot V)\co (A\ot \sigma_{V}^{A})\co  (\psi_{V}^{A}\ot V)\co 
(V\ot ((\nabla_{A\ot V}\co (\eta_{A}\ot V)))  $

\item[ ]$=(\mu_{A}\ot W)\co (A\ot \gamma)\co (\mu_A\ot V)\co (A\ot \sigma_{V}^{A})\co  (\psi_{V}^{A}\ot V)\co 
(V\ot ((\psi_{V}^{A}\co (V\ot \eta_{A})))  $

\item[ ]$= (\mu_{A}\ot W)\co (\mu_{A}\ot \gamma)\co  (A\ot \psi_{V}^{A})\co ( \sigma_{V}^{A}\ot \eta_{A})$

\item[ ]$=(\mu_{A}\ot W)\co (A\ot \gamma)\co  \nabla_{A\ot V}\co \sigma_{V}^{A}$

\item[ ]$= (\mu_{A}\ot W)\co (A\ot \gamma)\co \sigma_{V}^{A}, $

\end{itemize}

where the first and the seventh equalities follow by (\ref{gamma-theta-special}), the second one is a consequence of the associativity of $\mu_{A}$ and the left linearity of $\mu_{A\ot V}$, the third one relies on (\ref{gamma-theta-special}), and the fourth one follows by the left linearity of $\mu_{A\ot V}$ and $\nabla_{A\ot V}$. In the fifth one we used the definition of $\mu_{A\ot V}$, and the sixth one follows by (\ref{wmeas-wcp}). In the eighth and the tenth ones we applied the definition of $\nabla_{A\ot V}$,  the ninth one relies on the twisted condition (\ref{twis-wcp}) for ${\Bbb A}_{V}$, and the last one follows by the properties of $\sigma_{V}^{A}$.

\end{proof}

\begin{exemplo}
{\rm In this example we apply Theorem \ref{thm2-equiv} to the study of equivalent crossed products in the sense of Brzezi\'nski. First we recall from \cite{tb-crpr} the construction of Brzezi\'nski's crossed product in a strict monoidal category: Let $(A,\eta_{A},\mu_{A})$ be a monoid and $V$ an object equipped with a distinguished morphism $\eta_{V}:K\rightarrow V$. Then the object $A\ot V$ is a monoid with unit $\eta_{A}\ot \eta_{V}$ and whose product has the property $\mu_{A\ot V}\co (A\ot \eta_{V}\ot A\ot V)=\mu_{A}\ot V$, if and only if there exists two morphisms $\psi_{V}^{A}:V\ot A\rightarrow A\ot V$, $\sigma_{V}^{A}:V\ot V\rightarrow A\ot V$ satisfying  (\ref{wmeas-wcp}), the twisted condition (\ref{twis-wcp}), the cocycle condition (\ref{cocy2-wcp}) and  
\begin{equation}
\label{brz1} 
\psi_{V}^{A}\co (\eta_{V}\ot A)=A\ot \eta_{V}, 
\end{equation}
\begin{equation}
\label{brz2} 
\psi_{V}^{A}\co (V\ot \eta_{A})=\eta_{A}\ot V, 
\end{equation}
\begin{equation}
\label{brz3} 
\sigma_{V}^{A}\co (\eta_{V}\ot V)=\sigma_{V}^{A}\co (V\ot \eta_{V})=\eta_{A}\ot V.
\end{equation}
If this is the case, the product of $A\ot V$ is the one defined in (\ref{prod-todo-wcp}).  Note that Brzezi\'nski's crossed products are examples of weak crossed products where the associated idempotent is the identity, that is, $\nabla_{A\ot V}=id_{A\ot V}$. Also, in this case the preunit $\nu_{V}=\eta_{A}\ot \eta_{V}$ is a unit. 

In this setting (i) and (ii) of Theorem \ref{thm2-equiv} are the same and then this theorem can be enunciated in the following way: Let $(A\otimes V, \mu_{A\otimes V})$ and $(A\otimes W, \mu_{A\otimes W})$ be  Brzezi\'nski's crossed products  with distinguished morphism $\eta_{V}$ and $\eta_{W}$ respectively.  The following assertions are equivalent:
\begin{itemize}
\item[(i)] The  crossed products  $(A\otimes V, \mu_{A\otimes V})$ and $(A\otimes W, \mu_{A\otimes W})$ are equivalent.
\item[(ii)] There exist two morphisms 
$$\gamma:V\rightarrow A\ot W,\;\; \theta:W\rightarrow A\ot V$$
satisfying the conditions (\ref{gamma-theta-psi}), (\ref{gamma-theta-sigma}) and 
\begin{equation}
\label{gamma-theta-special-BRZ}
(\mu_{A}\ot V)\co (A\ot \theta)\co \gamma=\eta_{A}\ot V.
\end{equation} 
\end{itemize}
Note that (\ref{gamma-theta-special-BRZ}) is translation for  Brzezi\'nski's crossed products of (\ref{gamma-theta-special}).

In this context, (\ref{gamma-theta-idemp}) is trivial because $\nabla_{A\ot V}=id_{A\ot V}$ and $\nabla_{A\ot W}=id_{A\ot W}$. Moreover,  if we assume (\ref{gamma-theta-psi}), 
\begin{equation}
\label{gamma-theta-special-BRZ-b}
(\mu_{A}\ot W)\co (A\ot \gamma)\co \theta=\eta_{A}\ot W.
\end{equation} 
holds (see (iii) $\Rightarrow $ (ii) of the proof of Theorem \ref{thm2-equiv}).

On the other hand, by (\ref{gamma-theta-psi}), (\ref{gamma-theta-sigma}) and (\ref{gamma-theta-special-BRZ}), we obtain that 
\begin{equation}
\label{aux-brz}
(\mu_{A}\ot W)\co (A\ot ((\mu_{A}\ot W)\co (A\ot \sigma_{W}^{A})\co ((\gamma\co \eta_{V})\ot W)=id_{A\ot W}
\end{equation}
holds. Indeed:
\begin{itemize}

\item[ ]$\hspace{0.38cm} (\mu_{A}\ot W)\co (A\ot ((\mu_{A}\ot W)\co (A\ot \sigma_{W}^{A})\co ((\gamma\co \eta_{V})\ot W) $

\item[ ]$= (\mu_{A}\ot W)\co (A\ot ((\mu_{A}\ot W)\co (A\ot ((\mu_{A}\ot W)\co (A\ot \gamma)\co \mu_{A\ot V}\co (\theta\ot \theta)))\co ((\gamma\co \eta_{V})\ot W) $

\item[ ]$=(\mu_{A}\ot W)\co (A\ot ((\mu_{A}\ot W)\co (A\ot \gamma)\co \mu_{A\ot V}\co (((\mu_{A}\ot V)\co (A\ot \theta)\co \gamma\co \eta_{V})\ot \theta)  $

\item[ ]$=(\mu_{A}\ot W)\co (A\ot ((\mu_{A}\ot W)\co (A\ot \gamma)\co \mu_{A\ot V}\co (\eta_{A}\ot \eta_{V}\ot \theta)))$

\item[ ]$=(\mu_{A}\ot W)\co (A\ot ((\mu_{A}\ot W)\co (A\ot \gamma)\co \theta))  $

\item[ ]$=id_{A\ot W}$

\end{itemize}
where the first equality follows by (\ref{gamma-theta-sigma}), the second one by the associativity of $\mu_{A}$, the third one by (\ref{gamma-theta-special-BRZ}), the fourth one by the properties of $\mu_{A\ot V}$ and in the last one we applied (\ref{gamma-theta-special-BRZ-b}).

Therefore, composing in (\ref{aux-brz}) with $\eta_{A}\ot \eta_{W}$ we obtain 
\begin{equation}
\label{gamma-theta-preunit-b-BRZ}
\gamma\co \eta_{V}=\eta_{A}\ot \eta_{W}
\end{equation}
 Then, as a consequence of (\ref{gamma-theta-special-BRZ}) and (\ref{gamma-theta-preunit-b-BRZ}), we have 
\begin{equation}
\label{gamma-theta-preunit-BRZ}
\theta\co \eta_{W}=\eta_{A}\ot \eta_{V}.
\end{equation}
which is the translation to this setting of (\ref{gamma-theta-preunit}).

Therefore, for two Brzezi\'nski's crossed products $(A\otimes V, \mu_{A\otimes V})$ and $(A\otimes W, \mu_{A\otimes W})$ with distinguished morphism $\eta_{V}$ and $\eta_{W}$, the equivalence between them only depends of two morphisms $\gamma:V\rightarrow A\ot W,\;$ $\theta:W\rightarrow A\ot V$ satisfying (\ref{gamma-theta-psi}), (\ref{gamma-theta-sigma}) and (\ref{gamma-theta-special-BRZ}). 

This approach to the characterization of the notion of equivalence between Brzezi\'nski's  crossed products implies a  substantial  improvement of the result demonstrated by Panaite in Theorem 2.3 of \cite{Pan3} (see also Theorem 2.1 of \cite{Pan1}).  In this Theorem the author proved that two Brzezi\'nski's crossed products $(A\otimes V, \mu_{A\otimes V}^{1})$ and $(A\otimes V, \mu_{A\otimes V}^2)$, associated to the quadruples ${\Bbb A}_{V}^{1}=(A,V, \psi_{V}^{A,1}, \sigma_{V}^{A,1})$, ${\Bbb A}_{V}^{2}=(A,V, \psi_{V}^{A,2}, \sigma_{V}^{A,2})$, and with distinguished morphism $\eta_{V}^1$ and $\eta_{V}^2$, are equivalent if, and only if, there exist morphisms $\gamma$, $\theta:V\rightarrow A\ot V$ such that the following equalities are satisfied: 
\begin{equation}
\label{gamma-theta-psi-Pan}
\psi_{V}^{A,2}=(\mu_{A}\ot V)\co (\mu_{A}\ot \gamma)\co (A\ot \psi_{V}^{A,1})\co (\theta\ot A),
\end{equation}
\begin{equation}
\label{gamma-theta-sigma-Pan}
\sigma_{V}^{A,2}=(\mu_{A}\ot V)\co (A\ot \gamma)\co \mu_{A\ot V}^1\co (\theta\ot \theta) ,
\end{equation}
\begin{equation}
\label{gamma-theta-unit-Pan}
\theta\co \eta_{V}^2=\eta_{A}\ot \eta_{V}^1, \;\; \gamma\co \eta_{V}^1=\eta_{A}\ot \eta_{V}^2,
\end{equation}
\begin{equation}
\label{gamma-theta-special-BRZ-Pan}
(\mu_{A}\ot V)\co (A\ot \theta)\co \gamma=\eta_{A}\ot V,
\end{equation}
\begin{equation}
\label{gamma-theta-special-BRZ-1-Pan}
(\mu_{A}\ot V)\co (A\ot \gamma)\co \theta=\eta_{A}\ot V,
\end{equation}
\begin{equation}
\label{sup-11-Pan}
(\mu_{A}\ot V)\co (\mu_{A}\ot \sigma_{V}^{A,2})\co (A\ot \gamma\ot V)\co (\psi_{V}^{A,1}\ot V)\co (V\ot \gamma)=
(\mu_{A}\ot V)\co (A\ot \gamma)\co \sigma_{V}^{A,1}.
\end{equation}

Note that, using the results of our paper, by Proposition \ref{prop-equiv}, (\ref{sup-11-Pan}) follows from 
(\ref{gamma-theta-sigma-Pan}) and (\ref{gamma-theta-special-BRZ-Pan}). On the other hand, the equality (\ref{gamma-theta-special-BRZ-1-Pan}) is a consequence of (\ref{gamma-theta-psi-Pan}). Finally,  (\ref{gamma-theta-unit-Pan}) can be proved using  (\ref{gamma-theta-psi-Pan}),  (\ref{gamma-theta-sigma-Pan}) and (\ref{gamma-theta-special-BRZ-Pan}).

}
\end{exemplo}

\begin{exemplo}
{\rm 
Assume that ${\mathcal C}$ is symmetric with symmetry isomorphism $c$. Now we will see how we can use Theorem \ref{thm2-equiv} to improve the definition of equivalence of weak crossed products for weak Hopf algebras given in \cite{nrm-hhapl} and in \cite{ana1}. Let $H$ be a weak Hopf monoid with unit $\eta_{H}$, product $\mu_{H}$, counit $\varepsilon_{H}$, coproduct $\delta_{H}$, and antipode $\lambda_{H}$. Denote by $\delta_{H\ot H}$ the morphism $(H\ot c_{H,H}\ot H)\co (\delta_{H}\ot \delta_{H}).$
	
Recall from \cite{nrm-hhapl} and \cite{ana1} that if $(A, \varphi_A)$ is a left weak $H$-module monoid and $\sigma:H\otimes H\rightarrow A$ is a morphism, we say that the twisted condition is satisfied if:
	\begin{equation}\label{twisted-wha}
	\begin{array}{l}
	\mu_A\circ ((\varphi_A\circ (H\ot \varphi_A))\ot A)\circ (H\ot H\ot c_{A,A})\circ (((H\ot H \ot \sigma)\circ \delta_{H\ot H})\ot A) =\\
	\mu_A\circ (A\ot \varphi_A)\circ (((\sigma\ot \mu_H)\circ \delta_{H\ot H})\ot A).
	\end{array}
	\end{equation}
	We say that the cocycle condition holds if:
	\begin{equation}\label{cocy-wha}
	\begin{array}{l}
	\mu_A\circ (\varphi_A\ot \sigma)\circ (H\ot c_{H,A}\ot \mu_H)\circ (\delta_H\ot \sigma\ot H\ot H)\circ (H\ot \delta_{H\ot H}) = \\
	\mu_A\circ (\sigma\ot \sigma)\circ (H\ot H\ot \mu_H\ot H)\circ (\delta_{H\ot H}\ot H).
	\end{array}
	\end{equation}
	And the normal condition is satisfied if:
	\begin{equation}\label{normal-wha}
	\sigma\circ (\eta_H\ot H) = \sigma\circ (H\ot \eta_H) = u_{1}
	\end{equation}
where $u_{1}=\varphi_A\circ (H\ot \eta_A).$
	Under these conditions, we obtain a weak crossed product on $A\ot H$ with preunit $\nabla_{A\ot H}\circ (\eta_A\ot \eta_H)$ (see \cite{nrm-hhapl}, \cite{ana1}) and such that:
	\begin{equation}
	\label{psi-wha}
	\psi_{H}^A = (\varphi_A\ot H)\circ (H\ot c_{H,A})\circ (\delta_H\ot A)
	\end{equation}
	\begin{equation}\label{sigma-wha}
	\sigma_{H}^A = (\sigma\ot \mu_H) \circ \delta_{H\ot H}
	\end{equation}
	\begin{equation}\label{nabla-wha}
	\nabla_{A\ot H} = ((\mu_A\co (A\ot u_{1}))\ot H)\circ (A\ot \delta_{H}).
	\end{equation}
	In this context, when necessary, we will use the following notation: we will use $\psi_{H, \varphi_A}^A$ for the morphism defined in (\ref{psi-wha}) and $\nabla_{A\ot H}^{\varphi_A}$ for the morphism in (\ref{nabla-wha}). Moreover, we will denote the image of the idempotent by $A\times_{\varphi_A}^{\sigma} H$.
	
	Essentially, to define this weak crossed product, instead of considering any object $V$ in the category we take a weak Hopf algebra $H$. This fact permits to endow $A\ot H$ with a $H$-comodule structure given by
	\begin{equation}\label{comod-struc}
	\rho_{A\ot H} = A\ot \delta_H
	\end{equation}
	This comodule structure is inherited by the image of the idempotent $A\times_{\varphi_A}^{\sigma} H$, and it is given by:
	\begin{equation}\label{comod-struc-image}
	\rho_{A\times_{\varphi_A}^{\sigma} H} = (p_{A\ot H}\ot H)\circ \rho_{A\ot H}\circ i_{A\ot H}.
	\end{equation}
Moreover, the idempotent $\nabla_{A\ot H}^{\varphi_A}$ is a morphism of right $H$-comodules \cite{nrm-hhapl}, \cite{ana1}.	
	Taking into account these facts, we define the equivalence of weak crossed products for weak Hopf algebras as follows \cite{nrm-hhapl}, \cite{ana1}: Let $(A, \phi_A)$ be another left weak $H$-module monoid, and let $\beta:H\ot H \rightarrow A$ be a morphism that satisfies the twisted condition (\ref{twisted-wha}), the cocycle condition (\ref{cocy-wha}), and the normal condition (\ref{normal-wha}). Consider the quadruples $(A, H, \psi_{H, \varphi_A}^A, \sigma_{H}^{A})$ and $(A, H, \psi_{H, \phi_A}^A, \beta_{H}^{A})$, that induce two weak crossed products on $A\ot H$, and call $(A\ot H, \mu_{A\ot H}^{\varphi_A, \sigma})$ and $(A\ot H, \mu_{A\ot H}^{\phi_A, \beta})$ the corresponding weak crossed products. We say that they are equivalent if, and only if, there exists an isomorphism of monoids, left $A$-modules and right $H$-comodules (see \cite{nrm-hhapl}): 
	\[\alpha:A\times_{\varphi_A}^{\sigma}H\rightarrow A\times_{\phi_A}^{\beta}H. \]
	By virtue of Theorem \ref{thm2-equiv} we know that there exist two morphisms $S:A\ot H\to A\ot H$ and $T:A\ot H\to A\ot H$ of left $A$-modules that satisfy (\ref{preserv-preunit}), (\ref{preserv-product}) and (\ref{preserv-idemp}). Following the proof for $T$ to be multiplicative, it is possible to prove that $S$ is also multiplicative. Moreover, they satisfy (\ref{preserv-comp}). Finally, $T$ and $S$ are both morphisms of right $H$ comodules. Indeed, $T=i_{A\ot H}^{\phi_A}\circ \alpha\circ p_{A\ot H}^{\varphi_A}$ and $S=i_{A\ot H}^{\varphi_A}\circ \alpha^{-1}\circ p_{A\ot H}^{\phi_A}$, this is, both morphisms can be expressed as composition of morphisms of right $H$ comodules. Thus, we can use Definition \ref{wcp-equiv-def} and Theorem \ref{thm2-equiv} to obtain the equivalence of weak crossed products in the sense of \cite{ana1} as an example of our theory. Actually, the conditions exposed in the present paper are more general than the ones in \cite{ana1}, as we do not require $S$ and $T$ to be multiplicative. In a similar way, Theorem \ref{thm2-equiv} improves the conditions in \cite{nrm-hhapl} for two weak crossed products to be equivalent.
	
	Finally, observe that (iii) of Theorem \ref{thm2-equiv} is a generalization of the gauge equivalence defined in \cite{ana1}. Indeed, if we are in the weak Hopf algebra context, the morphisms $\gamma:H\to A\ot H$ and $\theta:H\to A\ot H$ are of right $H$-comodules. If we take $f = (A\ot \varepsilon_H)\circ \gamma$ and $g= (A\ot \varepsilon_H)\circ \theta$, we obtain morphisms $f,g:H\to A$ that give the gauge equivalence between the two crossed products. Observe that we can recover $\gamma$ and $\theta$ as
	\[\gamma = (f\ot H)\circ \delta_H,\;\; \theta = (g\ot H)\circ \delta_H.\]	
	Thus, composing in both sides of equality (\ref{gamma-theta-special}) we obtain $f\ast g = \varphi_A\circ (H\ot \eta_A)$, where $\ast$ denotes the usual convolution product. Moreover, if we compose with $A\ot \varepsilon_H$ in both sides of equality (\ref{gamma-theta-psi}) we obtain that
	\[\phi_A = \mu_A\circ (\mu_A\ot A)\circ (f\ot \varphi_A \ot g)\circ (H\ot H\ot c_{H,A})\circ (H\ot \delta_H\ot A)\circ (\delta_H\ot A). \]
	Now, composing with $(A\ot \varepsilon_H)$ in equality (\ref{gamma-theta-sigma}) we get:
	\[\beta = \mu_A\circ (\mu_A\ot A)\circ (\mu_A\ot A\ot A)\circ (f\ot (\varphi_A\circ (H\ot f))\ot \sigma \ot (g\circ \mu_H))\circ (\delta_H\ot H\ot \delta_{H\ot H})\circ \delta_{H\ot H}. \]
	This is, the pair $(f,g)$ satisfies the conditions of gauge equivalence given in \cite{ana1}.}
	\end{exemplo}

\begin{nota}
{\rm Note that we can obtain similar results about the equivalence between weak crossed products if we work int e mirror setting, i.e.,  if we use  quadruples ${\Bbb A}^{V}=(V,A,\psi_{A}^{V}, \sigma_{A}^{V})$ where $\psi_{A}^{V}:A\ot V\rightarrow V\ot A$ and $\sigma_{A}^{V}: V\ot V\rightarrow V\ot A$ satisfy the suitable conditions that define a weak crossed product on $V\ot A$ with preunit $\nu^{V}:K\rightarrow V\ot A$.}
\end{nota}

\section{The dual setting: equivalences between weak crossed coproducts} 

The theory of weak crossed coproducts  can be obtained from the corresponding ones for weak crossed
products defined in $A\otimes V$ or  $V\otimes A$ by dualization.  For the convenience of the reader we collect in the following paragraphs, the mirror version of the theory presented in the appendix of \cite{mra-preunit} for weak crossed coproducts.

Let $C$ be a comonoid and $V$ an object. Suppose
that there exists a morphism $\chi_{C}^{V}:V\otimes C\rightarrow C\otimes V$
such that the following equality holds:
\begin{equation}
\label{co-wmeas-wcp}
(C\ot \chi_{C}^{V})\circ (\chi_{C}^{V}\ot C)\circ
(V\ot \delta_{C})=( \delta_{C}\ot V)\circ \chi_{C}^{V}
\end{equation}

As a consequence the morphism $\Gamma_{V\otimes C}:V\otimes C\rightarrow V\otimes C$ defined by
\begin{equation}\label{co-idem-wcp}
\Gamma_{V\otimes C}=(\varepsilon_{C}\ot V\ot C)\circ
(\chi_{C}^{V}\ot C)\circ (V\ot \delta_{C}).
\end{equation}
 is  idempotent. Moreover, $\Gamma_{V\otimes C}$ is a  right $C$-comodule morphism
for  the right  coaction  $\rho_{V\otimes C}=V\ot \delta_{C}$.

Henceforth we will consider quadruples ${\Bbb C}^{V}=(C, V, \chi_{C}^{V}, \tau_{C}^{V})$ where $C$,
$V$ and $\chi_{C}^{V}$ satisfy the condition (\ref{co-wmeas-wcp}) and
$\tau_{C}^{V}:V\otimes C\rightarrow V\otimes V$ is a morphism in ${\mathcal C}$.
For the morphism $\Gamma_{V\otimes C}$ defined in (\ref{co-idem-wcp}) we denote by $V\Box C$ its  image and by
$i_{V\otimes C}:V\Box C\rightarrow V\otimes C$, $p_{V\otimes C}:V\otimes C\rightarrow V\Box C$ the associated injection and the projection respectively.

Following the ideas behind the theory of weak crossed products, we will set two properties that guarantee the coassociativity of certain weak crossed coproduct on $V\otimes C$.
We will say that a quadruple ${\Bbb C}^{V}$  satisfies the
cotwisted
condition if
\begin{equation}\label{co-twis-wcp}
(C\ot \tau_{C}^{V})\circ (\chi_{C}^{V}\ot C)\circ
(V\ot \delta_{C})=(\chi_{C}^{V}\ot V)\circ
(V\ot \chi_{C}^{V})\circ (\tau_{C}^{V}\ot C)\circ
(V\ot \delta_{C}),
\end{equation}
 and the cycle
condition holds if
\begin{equation}\label{co-cocy-wcp}
(V\ot \tau_{C}^{V})\circ (\tau_{C}^{V}\ot V)\circ
(V\ot \delta_{C})
\end{equation}
$$=(\tau_{C}^{V}\ot V)\circ (V\ot \chi_{C}^{V})\circ
(\tau_{C}^{V}\ot C)\circ (V\ot \delta_{C})
.$$

By virtue of the mirror version of Proposition A.9  of \cite{mra-preunit} we will consider from now on,
and without loss of generality, that
\begin{equation}
\tau_{C}^{V}\circ\Gamma_{V\otimes C} = \tau_{C}^{V}
\end{equation}
for all quadruples ${\Bbb C}^{V}$.

For a quadruple ${\Bbb C}^{V}$  define the coproduct
\begin{equation}\label{co-prod-todo-wcp}
\delta_{V\otimes  C} =(V\otimes \chi_{C}^{V}\otimes C)\circ
(\tau_{C}^{V}\ot \delta_{C})\circ (V\ot \delta_{C})
\end{equation}
and let $\delta_{V\Box C}$ be the coproduct
\begin{equation}
\label{co-prod-wcp} \delta_{V\Box C} = (p_{V\otimes C}\otimes
p_{V\otimes C})\circ\delta_{V\otimes C}\circ i_{V\otimes C}.
\end{equation}

If the cotwisted condition (\ref{co-twis-wcp}) and
the cycle condition (\ref{co-cocy-wcp}) hold, $\delta_{V\otimes C}$ is a coassociative
coproduct that it is conormalized with respect to $\Gamma_{V\otimes
C}$ (i.e. $\delta_{C\otimes V}\circ \Gamma_{V\otimes
C}=\delta_{V\otimes C}=(\Gamma_{V\otimes
C}\otimes \Gamma_{V\otimes
C})\circ \delta_{V\otimes C}$).  Also, 
\begin{equation}
\label{co-otra-prop} \delta_{V\otimes C}=(V\ot C\otimes \Gamma_{V\otimes C})\circ \delta_{V\otimes C}
\end{equation}
and therefore
\begin{equation}
\label{co-vieja-proof} \delta_{V\otimes C}=(\Gamma_{V\otimes
C}\otimes V\ot C)\circ \delta_{V\otimes C}.
\end{equation}

As a consequence $\delta_{V\Box C}$ is also a
coassociative coproduct (mirror version of Proposition A.10 of \cite{mra-preunit}). Under these circumstances we say that $(V\otimes C, \delta_{V\otimes C})$ is a weak crossed coproduct. Trivially, $\delta_{V\ot C}$ is right $C$-colinear  for  the right  coactions $\rho_{V\otimes C}$ and $\rho_{V\otimes C\otimes V\otimes C}=V\ot C\ot \rho_{V\otimes
C}$. Moreover, $\delta_{V\Box C}$ is right $C$-colinear for $\rho_{V\Box C}=(p_{V\ot C}\ot C)\co \rho_{V\otimes C}\co i_{V\ot C}$ and $\rho_{V\Box C\ot V\Box C}=V\Box C\ot \rho_{V\Box C}$.

Let $C$ be a comonoid and $V$ and object in ${\mathcal C}$. If $\Delta_{V\otimes
C}$ is a coassociative coproduct defined in
$V\otimes C$ with  precounit $\upsilon^{V}: V\otimes C\rightarrow K$ , i.e. a morphism
satisfying
\begin{equation}
(\upsilon^{V}\ot V\ot C)\circ \Delta_{V\otimes C}=(V\ot C\ot \upsilon^{V})\circ 
\Delta_{V\otimes C}=(((\upsilon^{V}\otimes \upsilon^{V})\co \Delta_{V\otimes C}) \ot V\ot C)\circ
\Delta_{V\otimes C},
\end{equation}
we obtain that the image of the idempotent morphism $$\Gamma_{V\otimes
C}^{\upsilon^{V}}=(\upsilon^{V}\ot V\otimes C)\circ \Delta_{V\otimes C}:V\otimes
C\rightarrow V\otimes C,$$
denoted by $V\Box C$, is a comonoid with coproduct
$$\Delta_{V\Box C}=(p_{V\otimes C}^{\upsilon^{V}}\otimes p_{V\otimes C}^{\upsilon^{V}})\circ
\Delta_{V\otimes C}\circ i_{V\otimes C}^{\upsilon^{V}},$$
and counit $\varepsilon_{V\Box C}=\upsilon^{V}\circ i_{V\otimes C}^{\upsilon^{V}}$, where
$p_{V\otimes C}^{\upsilon^{V}}$ and $i_{V\otimes C}^{\upsilon^{V}}$ are the injection and
the projection associated to the idempotent  (see the mirror version of Proposition A.5 of \cite{mra-preunit}).

If moreover, $\Delta_{V\otimes C}$ is right $C$-colinear for the coactions
$\rho_{V\otimes C}$ and $\rho_{V\otimes C\otimes V\otimes C}$  and conormalized with respect to $\Gamma_{V\otimes
C}^{\upsilon^{V}}$, the morphism
\begin{equation}
\label{gamma-nu}
\omega_{\upsilon^{V}}:V\otimes C\rightarrow  C,\; \omega_{\upsilon^{V}} =
(\upsilon^{V}\ot C)\circ (V\ot \delta_{C})
\end{equation}
is comultiplicative and right $C$-colinear for $\rho_{C}=\delta_{C}$. Although
$\omega_{\upsilon^{V}}$
is not a comonoid morphism, because $V\otimes C$ is not a
comonoid, we have that $\varepsilon_{C}\circ \omega_{\upsilon^{V}} =
\upsilon^{V}$, and, as a consequence, the morphism $\bar{\omega}_{\upsilon^{V}}
=\omega_{\upsilon^{V}}\circ  i_{V\otimes C}^{\upsilon^{V}}
:V\Box C\rightarrow C$ is a comonoid morphism.

In the following theorem we give a characterization of weak
crossed coproducts with precounit. This result is the dual version of the one proved for weak crossed products in the first section (see Theorem A.13. of \cite{mra-preunit} for the mirror version):

\begin{teo}
\label{co-thm1-wcp} Let $C$ be a comonoid, $V$ an object and
$\Delta_{V\otimes C}:V\otimes C\rightarrow  V\otimes C\otimes V\otimes C$ a
morphism of right $C$-comodules for  the right  coaction $\rho_{V\otimes C}$ and $\rho_{V\otimes C\otimes V\otimes C}$.

Then the following statements are equivalent:
\begin{itemize}
\item[(i)] The coproduct $\Delta_{V\otimes C}$ is coassociative with precounit
$\upsilon^{V}$ and normalized with respect to $\Gamma_{V\otimes C}^{\upsilon^{V}}.$

\item[(ii)] There exist morphisms $\chi_{C}^{V}:V\otimes C\rightarrow C\otimes
V$, $\tau_{C}^{V}:V\otimes C\rightarrow V\otimes V$ and $\upsilon^{V}:V\otimes
C\rightarrow K$ such that if $\delta_{V\otimes C}$ is the coproduct defined
in (\ref{co-prod-todo-wcp}), the pair $(V\otimes C, \delta_{V\otimes
C})$ is a weak crossed coproduct with $\Delta_{V\otimes C} =
\delta_{V\otimes C}$ satisfying:
\begin{equation}\label{co-pre1-wcp}
( \upsilon^{V}\ot V)\circ (V\ot \chi_{C}^{V})\circ (\tau_{C}^{V}\ot C)\circ (V\ot \delta_{C})=(V\ot \varepsilon_{C})\circ \Gamma_{V\otimes C}
\end{equation}
\begin{equation}\label{co-pre2-wcp}
(V\ot \upsilon^{V})\circ (\tau_{C}^{V}\ot C)\circ (V\ot \delta_{C})=(V\ot \varepsilon_{C})\circ \Gamma_{V\otimes C}\end{equation}
\begin{equation}\label{co-pre3-wcp}
(C\ot \upsilon^{V})\circ (\chi_{C}^{V}\ot C)\circ
(V\ot\delta_{C}) = \omega_{\upsilon^{V}},
\end{equation}
\end{itemize}
where $\omega_{\upsilon^{V}}$ is the morphism defined in
(\ref{gamma-nu}). In this case $\upsilon^{V}$ is a precounit for
$\delta_{V\otimes C}$, and  the idempotent morphism of the weak crossed
coproduct $\Gamma_{V\otimes C}$ is the idempotent $\Gamma_{V\otimes
C}^{\nu^{V}}$
\end{teo}

\begin{nota}
{\rm Note that in the proof of the previous
theorem, we obtain that  
\begin{equation}\label{co-fi-wcp}
\chi_{C}^{V} = (\omega_{	\upsilon^{V}}\ot V\ot \varepsilon_{C})\co \delta_{V\ot C},
\end{equation}
\begin{equation}\label{co-sigma-wcp}
\tau_{C}^{V} = (V\ot \varepsilon_{C}\ot V\ot \varepsilon_{C})\co \delta_{V\ot C}, 
\end{equation}
hold. Also, by (\ref{co-pre3-wcp}), we have 
\begin{equation}\label{co-preunit-idemp}
\upsilon^{V}\co \Gamma_{V\otimes C}=\upsilon^{V}.
\end{equation}

}
\end{nota}

\begin{defin}\label{wcp-pre-def}
{\rm We will say that a weak crossed coproduct $(C\otimes V, \delta_{C\otimes V})$ is a weak crossed coproduct with precounit 
$\upsilon{V}:V\ot C\rightarrow K$, if  (\ref{co-pre1-wcp}), (\ref{co-pre2-wcp}) and (\ref{co-pre3-wcp}) hold.}
\end{defin}

As a corollary of the previous result we have:

\begin{corol}\label{co-corol-wcp}
If $(C\otimes V, \delta_{C\otimes V})$ is a weak crossed coproduct
with precounit $\upsilon$, then $C\Box V$ is a coalgebra with the
coproduct defined in (\ref{co-prod-wcp}) and  counit
$\varepsilon_{C\Box V}=\upsilon^V\circ i_{C\otimes V}$.
\end{corol}

In the following definition we introduce the notion of equivalent weak crossed coproducts. 

\begin{defin}\label{wcp-equiv-def}
{\rm Let  $(V\otimes C, \delta_{V\otimes C})$ and $(W\otimes C, \delta_{W\otimes C})$ be weak crossed coproducts  with precounits  $\upsilon^{V}$ and $\upsilon^{W}$ respectively.  We will say that $(V\otimes C, \delta_{V\otimes C})$ and $(W\otimes C, \delta_{W\otimes C})$ are equivalent if there exists a comonoid isomorphism $\beta: V\Box C\rightarrow W\Box C$ of right $C$-comodules for the actions $\rho_{V\Box C}$ and $\rho_{W\Box C}$.}
\end{defin}

Then, we are in an ideal position to enunciate (the proof is dual to the one used for Theorem \ref{thm2-equiv}) the result that characterizes equivalent weak crossed coproducts.

\begin{teo}
\label{co-thm2-equiv}
Let $(V\otimes C, \delta_{V\otimes C})$ and $(W\otimes C, \delta_{W\otimes C})$ be weak crossed coproducts  with precounits  $\upsilon^{V}$ and $\upsilon^{W}$ respectively.  The following assertions are equivalent:

\item[(i)] The weak crossed coproducts  $(V\otimes C, \delta_{V\otimes C})$ and $(W\otimes C, \delta_{W\otimes C})$ are equivalent.

\item[(ii)] There exist two morphisms 
$$P:V\ot C\rightarrow W\ot C,\;\; R:W\ot C\rightarrow V\ot C$$
of right $C$-comodules, for the  coactions $\rho_{V\otimes C}$ and $\rho_{W\otimes C}$,  satisfying the conditions 
\begin{equation}
\label{co-preserv-preunit}
\upsilon^{W}\co P=\upsilon^{V},
\end{equation}
\begin{equation}
\label{co-preserv-product}
\delta_{W\otimes C}\co P=(P\ot P)\co \delta_{V\otimes C},
\end{equation}
\begin{equation}
\label{co-preserv-idemp}
R\co P=\Gamma_{V\ot C},\;\; P\co R=\Gamma_{W\ot C}.
\end{equation}

\item[(iii)] There exist two morphisms 
$$\pi:W\ot C\rightarrow V,\;\; \zeta:V\ot C\rightarrow W$$
satisfying the conditions 
\begin{equation}
\label{co-gamma-theta-preunit}
\upsilon^{W}\co (\zeta\ot C)\co (V\ot \delta_{C})=\upsilon^{V},
\end{equation}
\begin{equation}
\label{co-gamma-theta-idemp}
\zeta=\zeta\co \Gamma_{V\ot C},
\end{equation}
\begin{equation}
\label{co-gamma-theta-psi}
\chi_{C}^{W}=(C\ot \zeta)\co (\chi_{C}^{V}\ot C)\co (\pi\ot \delta_{C})\co (W\ot \delta_{C}),
\end{equation}
\begin{equation}
\label{co-gamma-theta-sigma}
\tau_{C}^{W}=(\zeta\ot \zeta)\co \delta_{V\ot C}\co (\pi\ot C)\co (W\ot \delta_{C}) ,
\end{equation}
\begin{equation}
\label{co-gamma-theta-special}
\pi\co (\zeta\ot C)\co (V\ot \delta_{C})=(V\ot \varepsilon_{C})\co \Gamma_{V\ot C}.
\end{equation}
\end{teo}

\begin{nota}
{\rm Note that we can obtain similar results about the equivalence between weak crossed coproducts if we work in the mirror setting, i.e.,  if we use  quadruples ${\Bbb C}_{V}=(V,C,\chi_{V}^{C}, \tau_{V}^{C})$ where $\chi_{V}^{C}:C\ot V\rightarrow V\ot C$ and $\tau_{V}^{C}: C\ot V\rightarrow V\ot V$ satisfy the suitable conditions that define a weak crossed coproduct on $C\ot V$ with preunit $\upsilon_{V}:C\ot V\rightarrow K$.}
\end{nota}

\section{The mixed case: equivalences between weak crossed biproducts}

Following Definition 2.2 of \cite{mra-proj} we introduce the notions of weak crossed biproduct and the one of weak crossed biproduct with preunit and precounit.  Note that the following definition is the dual version of Definition 2.2 of \cite{mra-proj}. 

\begin{defin}
\label{biproduct-defin}
{\rm 
Let $A$ be a monoid and $C$ be a comonoid. We will say that ${\Bbb A}{\Bbb C}=(A\ot C, \mu_{A\ot C},\delta_{A\ot C})$ is a weak crossed biproduct if:
\begin{itemize}
\item[(i)] The pair $(A\ot C, \mu_{A\ot C})$ is a weak crossed product.
\item[(ii)] The pair $(A\ot C, \delta_{A\ot C})$ is a weak crossed coproduct.
\item[(iii)] If $\nabla_{A\ot C}$ is the idempotent associated to $(A\ot C, \mu_{A\ot C})$, and $\Gamma_{A\ot C}$ is the idempotent associated to $(A\ot C, \delta_{A\ot C})$, the identity $\nabla_{A\ot C}=\Gamma_{A\ot C}$ holds.
\end{itemize}

In this case, $\nabla_{A\ot C}$  will be called the idempotent associated to  ${\Bbb A}{\Bbb C}$.

The triple ${\Bbb A}{\Bbb C}$ is a weak crossed biproduct with preunit $\nu_{C}$ and precounit $\upsilon^{A}$, if ${\Bbb A}{\Bbb C}$ is a weak crossed biproduct, $(A\ot C, \mu_{A\ot C})$ is a weak crossed product with preunit $\nu_{C}$, $(A\ot C, \delta_{A\ot C})$  is a weak crossed coproduct with precounit $\upsilon^{A}$, and the following identities are satisfied:
\begin{equation}
\label{3-1}
\eta_{A}=(A\ot \varepsilon_{C})\co \nu_{C},\;\; \varepsilon_{C}=\upsilon^{A}\co (\eta_{A}\ot C).
\end{equation}
}
\end{defin}

In the weak setting, the definition of equivalent weak crossed biproducts with preunit is introduced mixing the corresponding ones for weak crossed products and coproducts.

\begin{defin}
\label{biprod-equiv}
{\rm Let ${\Bbb A}{\Bbb C}_{1}=(A\ot C, \mu_{A\ot C}^{1},\delta_{A\ot C}^{1})$,   ${\Bbb A}{\Bbb C}_{2}=(A\ot C, \mu_{A\ot C}^{2},\delta_{A\ot C}^{2})$ be weak crossed biproducts with preunits $\nu_{C}^{1}$, $\nu_{C}^{2}$, and precounits $\upsilon^{A,1}$, $\upsilon^{A,2}$ respectively. ${\Bbb A}{\Bbb C}_{1}$ and ${\Bbb A}{\Bbb C}_{2}$ will be called equivalent if there exists an isomorphism $\alpha:A\times_{1}C\rightarrow A\times_{2}C$ of monoids, comonoids, left $A$-modules and right $C$-comodules, where $A\times_{1}C$ is the image of the idempotent associated to ${\Bbb A}{\Bbb C}_{1}$,  $A\times_{2}C$ is the image of the idempotent associated to ${\Bbb A}{\Bbb C}_{2}$, and the actions and the coactions are $\varphi_{A\times_{k}C}=p^{k}_{A\ot C}\co \varphi_{A\ot C}\co i^{k}_{A\ot C}$, $\rho_{A\times_{k}C}=(p^{k}_{A\ot C}\ot C)\co \rho_{A\ot C}\co 
\co i^{k}_{A\ot C}$, $k\in \{1,2\}$.

Note that, when we work with Brzezi\'nski's crossed products and coproducts, the previous definition is the one used by Panaite for equivalent product bialgebras in \cite{Pan3}.

}
\end{defin}

Then, using the proofs of the main theorems of the previous sections, we obtain the characterization of weak crossed biproducts with unit and precounit.

\begin{teo}
\label{co-thm2-equiv}
Let ${\Bbb A}{\Bbb C}_{1}$ and ${\Bbb A}{\Bbb C}_{2}$ be weak crossed biproducts  with preunits 
$\nu_{C}^{1}$, $\nu_{C}^{2}$, and precounits $\upsilon^{A,1}$, $\upsilon^{A,2}$ respectively.  The following assertions are equivalent:

\begin{itemize}

\item[(i)] The weak crossed biproducts ${\Bbb A}{\Bbb C}_{1}$ and ${\Bbb A}{\Bbb C}_{2}$ are equivalent.

\item[(ii)] There exist two morphisms 
$$T:A\ot C\rightarrow A\ot C,\;\; S:A\ot C\rightarrow A\ot C$$
of left $A$-modules for the  action $\varphi_{A\ot C}$ and right $C$-comodules for the  coaction $\rho_{A\otimes C}$,  satisfying the conditions 
\begin{equation}
\label{bi-preserv-preunit}
T\co \nu_{C}^1=\nu_{C}^2,
\end{equation}
\begin{equation}
\label{bi-co-preserv-preunit}
\upsilon^{A,2}\co T=\upsilon^{A,1},
\end{equation}
\begin{equation}
\label{bi-preserv-product}
T\co \mu_{A\otimes C}^1=\mu_{A\otimes C}^2\co (T\ot T),
\end{equation}
\begin{equation}
\label{bi-co-preserv-product}
\delta_{A\otimes C}^2\co T=(T\ot T)\co \delta_{A\otimes C}^1,
\end{equation}
\begin{equation}
\label{bi-preserv-idemp}
S\co T=\nabla_{A\ot C}^1,\;\;  T\co S=\nabla_{A\ot C}^2.
\end{equation}

\item[(iii)] There exist four morphisms 
$$\gamma, \theta:C\rightarrow A\ot C,\;\; \pi, \zeta:A\ot C\rightarrow C$$
satisfying the conditions 
\begin{equation}
\label{bi-gamma-theta-preunit}
\nu_{C}^1=(\mu_{A}\ot V)\co (A\ot \theta)\co \nu_{C}^2,
\end{equation}
\begin{equation}
\label{bi-co-gamma-theta-preunit}
\upsilon^{A,2}\co (\zeta\ot C)\co (A\ot \delta_{C})=\upsilon^{A,1},
\end{equation}
\begin{equation}
\label{bigamma-theta-idemp}
\theta=\nabla_{A\ot C}^1\co \theta,  \;\; \zeta=\zeta\co \nabla_{A\ot C}^1,
\end{equation}
\begin{equation}
\label{bi-gamma-theta-psi}
\psi_{C}^{A,2}=(\mu_{A}\ot C)\co (\mu_{A}\ot \gamma)\co (A\ot \psi_{C}^{A,1})\co (\theta\ot A),
\end{equation}
\begin{equation}
\label{bi-gamma-theta-sigma}
\sigma_{C}^{A,2}=(\mu_{A}\ot C)\co (A\ot \gamma)\co \mu_{A\ot C}^1\co (\theta\ot \theta) ,
\end{equation}
\begin{equation}
\label{bi-gamma-theta-special}
(\mu_{A}\ot C)\co (A\ot \theta)\co \gamma=\nabla_{A\ot C}^1\co (\eta_{A}\ot C).
\end{equation}
\begin{equation}
\label{bi-co-gamma-theta-psi}
\chi_{C}^{A,2}=(C\ot \zeta)\co (\chi_{C}^{A,1}\ot C)\co (\pi\ot \delta_{C})\co (A\ot \delta_{C}),
\end{equation}
\begin{equation}
\label{bi-co-gamma-theta-sigma}
\tau_{C}^{A,2}=(\zeta\ot \zeta)\co \delta_{A\ot C}^1\co (\pi\ot C)\co (A\ot \delta_{C}) ,
\end{equation}
\begin{equation}
\label{bi-co-gamma-theta-special}
\pi\co (\zeta\ot C)\co (A\ot \delta_{C})=(A\ot \varepsilon_{C})\co \nabla_{A\ot C}^1.
\end{equation}

\end{itemize}
\end{teo}

\begin{nota}
{\rm Note that, as a particular instance of the previous theorem, we obtain the characterization proposed by Panaite in Theorem 3.6 of \cite{Pan3} for equivalent cross product bialgebras.
}
\end{nota}

\section*{Acknowledgements}
The  first and second authors were supported by  Ministerio de Econom\'{\i}a y Competitividad of Spain (European Feder support included). Grant MTM2013-43687-P: Homolog\'{\i}a, homotop\'{\i}a e invariantes categ\'oricos en grupos y \'algebras no asociativas.  The  first author also was  supported by Consellería de Cultura, Education e Ordenación Universitaria. Grant GRC2013-045.

\end{document}